\documentclass[12pt]{amsart}
\newtheorem{thm}{Theorem}[section]

\newtheorem{lem}[thm]{Lemma}
\newtheorem{prop}[thm]{Proposition}
\theoremstyle{definition}
\newtheorem{defn}[thm]{Definition}
\theoremstyle{remark}
\newtheorem{rem}[thm]{Remark}
\numberwithin{equation}{section}

\def\bb{{\mathbb B}}
\def\bc{{\mathbb C}}

\def\bm{{\mathbb M}}
\def\bn{{\mathbb N}}

\def\br{{\mathbb R}}

\def\a{\alpha}
\def\b{\beta}
\def\g{\gamma}  
  \def\D{\Delta}
\def\ve{\varepsilon}

\def\l{\lambda}

\def\p{\psi}

\def\s{\sigma} 
\def\t{\tau}
\def\f{\varphi}

\def\w{\omega} 

\def\tr{\mathop{\rm Tr}}
\def\id{{\bf 1}\!\!{\rm I}}
\def\fb{{\mathbf{f}}}

\def\pb{{\mathbf{p}}}
\def\xb{{\mathbf{x}}}
\def\bx{{\mathbf{b}}}

\def\qb{{\mathbf{q}}}

\def\D{\Delta}

\def\wb{{\mathbf{w}}}

\def\o{\otimes}

\def\a{\alpha}

\begin{document}
\title[On dynamics of quadratic operators]
{On Kadison-Schwarz type quantum quadratic operators\\ on
$\bm_2(\mathbb{C})$}

\author{Farrukh Mukhamedov}
\address{Farrukh Mukhamedov\\
 Department of Computational \& Theoretical Sciences\\
Faculty of Science, International Islamic University Malaysia\\
P.O. Box, 141, 25710, Kuantan\\
Pahang, Malaysia} \email{{\tt far75m@yandex.ru}, {\tt
farrukh\_m@iium.edu.my}}

\author{Abduaziz Abduganiev}
\address{Abduaziz Abduganiev\\
 Department of Computational \& Theoretical Sciences\\
Faculty of Science, International Islamic University Malaysia\\
P.O. Box, 141, 25710, Kuantan\\
Pahang, Malaysia} \email{{\tt azizi85@yandex.ru}}

\begin{abstract}

In the present paper we study description of Kadison-Schwarz type
quantum quadratic operators acting from $\bm_2(\mathbb{C})$ into
$\bm_2(\mathbb{C})\o\bm_2(\mathbb{C})$. Note that such kind of
operator is a generalization of quantum convolution. By means of
such a description we provide an example of q.q.o. which is not a
Kadision-Schwartz operator. Moreover, we study dynamics of an
associated nonlinear (i.e. quadratic) operators acting on the state
space of $\bm_2(\mathbb{C})$.

\end{abstract}

\maketitle

\section{Introduction}

It is known that one of the main problems of quantum information
is characterization of positive and completely positive maps on
$C^*$-algebras. There are many papers devoted to this problem (see
for example \cite{Choi,MM1,RSW,St}). In the literature the
completely positive maps have proved to be of great importance in
the structure theory of C$^*$-algebras. However, general positive
(order-preserving) linear maps are very intractable\cite{MM1,Ma}.
It is therefore of interest to study conditions stronger than
positivity, but weaker than complete positivity. Such a condition
is called {\it Kadison-Schwarz property}, i.e a map $\phi$
satisfies the Kadison-Schwarz property if $\phi(a)^*\phi(a)\leq
\phi(a^*a)$ holds for every $a$. Note that every unital completely
positive map satisfies this inequality, and a famous result of
Kadison states that any positive unital map satisfies the
inequality for self-adjoint elements $a$. In \cite{Rob} relations
between $n$-positivity of a map $\phi$ and the Kadison-Schwarz
property of certain map is established. Certain relations between
complete positivity, positivite and the Kadison-Schwarz property
have been considered in \cite{Bh},\cite{BS},\cite{BD}.  Some
spectral and ergodic properties of Kadison-Schwarz maps were
investigated in \cite{Gr1,Gr2,Rob1}.

In \cite{MAA} we have studied quantum quadratic operators (q.q.o.),
i.e. maps from $\bm_2(\mathbb{C})$ into
$\bm_2(\mathbb{C})\o\bm_2(\mathbb{C})$, with the Kadison-Schwarz
property. It is found some necessary conditions for the
trace-preserving quadratic operators to be the Kadison-Schwarz ones.
Since trace-preserving maps arise naturally in quantum information
theory (see e.g. \cite{N}) and other situations in which one wishes
to restrict attention to a quantum system that should properly be
considered a subsystem of a larger system with which it interacts.
Note that in \cite{GM1,GM2} quantum quadratic operators acting on a
von Neumann algebra were defined and studied. Certain ergodic
properties of such operators were studied in \cite{M2,M3} (see for
review \cite{GMR}). In the present paper we continue our
investigation, i.e. we are going to study further properties of
q.q.o. with Kadison-Schwarz property. We will provide an example of
q.q.o. which is not a Kadision-Schwarz operator, and study its
dynamics. We should stress that q.q.o. is a generalization of
quantum convolution (see \cite{W}). Some dynamical properties of
quantum convolutions were investigated in \cite{FS}.

Note that a description of bistochastic Kadison-Schwarz mappings
from $\bm_2(\mathbb{C})$ into $\bm_2(\mathbb{C})$ has been provided
in \cite{MA}.

\section{Preliminaries}

In what follows, by $\bm_2(\bc)$ we denote an algebra of $2\times 2$
matrices over complex filed $\bc$. By $\bm_2(\bc)\o \bm_2(\bc)$ we
mean tensor product of $\bm_2(\bc)$ into itself. We note that such a
product can be considered as an algebra of $4\times 4$ matrices
$\bm_4(\bc)$ over $\bc$. In the sequel $\id$ means an identity
matrix, i.e. $ \id = \left( \begin{array}{cc} 1 & 0 \\ 0 & 1
\end{array} \right) $. By $S(\bm_2(\bc))$ we denote the set of all
states (i.e. linear positive functionals which take value 1 at
$\id$) defined on $\bm_2(\bc)$.

\begin{defn}\label{qqso} A linear operator $\D:\bm_2(\bc)\to \bm_2(\bc)\o\bm_2(\bc)$ is said to be
\begin{enumerate}
\item[(a)] -- a {\it quantum quadratic operator (q.q.o.)} if it
satisfies the following conditions:
\begin{enumerate}
    \item[(i)] unital, i.e. $\D\id=\id\o\id$;

  \item[(ii)] $\D$ is positive, i.e.
$\D x\geq 0$ whenever $x\geq 0$;
\end{enumerate}
\item[(b)] -- a {\it Kadison-Schwarz operator (KS)} if it satisfies
\begin{equation}\label{KS}
\D(x^*x)\geq\D(x^*)\D(x) \ \ \textrm{for all} \ x\in\bm_2(\bc) .
\end{equation}
\end{enumerate}
\end{defn}

One can see that if $\D$ is unital and KS operator, then it is a
q.q.o. A state  $h\in S(\bm_2(\bc))$ is called {\it a Haar state}
for a q.q.o. $\D$ if for every $x\in\bm_2(\bc)$ one has
\begin{equation}\label{Haar}
(h\o id)\circ \D(x)=(id\o h)\circ\D(x)=h(x)\id.
\end{equation}

\begin{rem}\label{qg} Note that if a quantum convolution $\D$ on
$\bm_2(\bc)$ becomes a $*$-homomorphic map with a condition
$$
\overline{\textrm{Lin}}((\id\o
\bm_2(\bc))\D(\bm_2(\bc)))=\overline{\textrm{Lin}}((\bm_2(\bc)\o\id)\D(\bm_2(\bc)))=\bm_2(\bc)\o\bm_2(\bc)
$$
then a pair $(\bm_2(\bc),\D)$ is called a {\it compact quantum
group} \cite{W}. It is known \cite{W} that for given any compact
quantum group there exists a unique Haar state w.r.t. $\D$.
\end{rem}

\begin{rem} Let $U:\bm_2(\bc)\o\bm_2(\bc)\to \bm_2(\bc)\o\bm_2(\bc)$
be a linear operator such that $U(x\o y)=y\o x$ for all $x,y\in
\bm_2(\bc)$. If a q.q.o. $\D$ satisfies $U\D=\D$, then $\D$ is
called a {\it quantum quadratic stochastic operator}. Such a kind of
operators were studied and investigated in \cite{M2}.
\end{rem}

Each q.q.o. $\D$ defines a conjugate operator
$\D^*:(\bm_2(\bc)\o\bm_2(\bc))^*\rightarrow \bm_2(\bc)^*$  by
\begin{equation}\label{cqo}
\D^*(f)(x)=f(\D x), \ \ f\in (\bm_2(\bc)\o \bm_2(\bc))^*, \ x\in
\bm_2(\bc).
\end{equation}

One can define an operator $V_\D$  by
\begin{equation}\label{qo} V_\D(\f)=\D^*(\f\o\f), \ \
\f\in S(\bm_2(\bc)),
\end{equation}
which is called a {\it quadratic operator (q.c.)}. Thanks to the
conditions (a) (i),(ii) of Def. \ref{qqso} the operator $V_\D$
maps $S(\bm_2(\bc))$ to $S(\bm_2(\bc))$.

\section{Quantum quadratic operators with Kadison-Schwarz property
on $\bm_2(\bc)$}

In this section we are going to describe quantum quadratic operators
on $\bm_2(\bc)$ as well as find necessary conditions for such
operators to satisfy the Kadison-Schwarz property.

Recall \cite{BR} that the identity and Pauli matrices $\{ \id,
\sigma_1, \sigma_2, \sigma_3 \}$ form a basis for $\bm_2(\bc)$,
where
\begin{eqnarray*}
\sigma_1 = \left( \begin{array}{cc} 0 & 1 \\ 1 & 0 \end{array}
\right)~~ \sigma_2 = \left( \begin{array}{cc} 0 & -i \\ i & 0
\end{array} \right)~~ \sigma_3 = \left( \begin{array}{cc} 1 & 0 \\
0 & -1 \end{array} \right).
\end{eqnarray*}

In this basis every matrix $x\in\bm_2(\bc)$ can  be written as $x =
w_0\id + \wb{\bf \sigma}$ with $w_0\in\bc$, $\wb =(w_1,w_2,w_3)\in
\bc^3$, here $\wb\s=w_1\s_1+w_2\s_2+w_3\s_3$.

\begin{lem}\label{m2}\cite{RSW} The following assertions hold true:
\begin{enumerate}
\item[(a)] $x$ is self-adjoint iff  $w_0,\wb$  are reals; \item[(b)]
$\tr(x) = 1$ iff $w_0 =0.5$, here $\tr$ is the trace of a matrix
$x$;
\item[(c)] $x
> 0$ iff $\|\wb\|\leq w_0$, where
$\|\wb\|=\sqrt{|w_1|^2+|w_2|^2+|w_3|^2}$.
\end{enumerate}
\end{lem}

Note that any state $\f\in S(\bm_2(\bc))$ can be represented by
\begin{equation}\label{state}
{\f}(w_0\id + \wb\sigma)=w_0+\langle\wb,{\mathbf{f}}\rangle, \ \
\end{equation}
where ${\mathbf{f}}=(f_1,f_2,f_3)\in\br^3$ with
$\|{\mathbf{f}}\|\leq 1$. Here as before $\langle\cdot,\cdot\rangle$
stands for the scalar product in $\bc^3$. Therefore, in the sequel
we will identify a state $\varphi$ with a vector $\fb\in \br^3$.

In what follows by $\t$ we denote a normalized trace, i.e.
$\t(x)=\frac{1}{2}\tr(x)$, $x\in \bm_2(\bc)$,

Let  $\Delta:\bm_2(\bc)\rightarrow \bm_2(\bc) \otimes \bm_2(\bc)$ be
a q.q.o. with a Haar state $\t$. Then one has
$$
\tau\otimes \tau(\Delta x)=\tau(\tau\otimes
id)(\Delta(x))=\tau(x)\t(\id)=\t(x), \ \ \ x\in \bm_2(\bc),
$$
which means that $\t$ is an invariant state for $\D$.

 Let
us write the operator $\Delta $ in terms of a basis in
$\bm_2(\bc)\o\bm_2(\bc)$ formed by the Pauli matrices. Namely,
\begin{eqnarray*}
&& \Delta \id=\id\otimes \id; \\\label{ddd1} && \Delta
(\sigma_i)=b_i(\id\otimes \id)+\overset{3}{\underset{j=1}{\sum
    }}b_{ji}^{(1)}(\id\otimes \sigma_j)+\overset{3}{\underset{j=1}{\sum
    }}b_{ji}^{(2)}(\sigma_j \otimes \id)+\overset{3}{\underset{m,l=1}{\sum
    }}b_{ml,i}(\sigma_m \otimes \sigma_l), \ \  i=1,2,3,
\end{eqnarray*}
where $b_{i},b_{ij}^{(1)},b_{ij}^{(2)},b_{ijk}\in \bc$,
($i,j,k\in\{1,2,3\}$).

 One can prove the following

\begin{thm}\cite[Proposition 3.2]{MAA}\label{qc1}
Let $\D:\bm_2(\bc)\to \bm_2(\bc)\o\bm_2(\bc) $ be a q.q.o. with a
Haar state $\t$, then it has the following form:
\begin{equation}\label{D3}
\D(x)=w_0\id\otimes\id+\sum_{m,l=1}^3\langle\bx_{ml},\overline{\wb}\rangle\sigma_m\otimes\sigma_l,
\end{equation}
where $x=w_0+\wb\s$,
$\bx_{ml}=(b_{ml,1},b_{ml,2},b_{ml,3})\in\br^3$,
$m,n,k\in\{1,2,3\}$.
\end{thm}

Let us turn to the positivity of $\D$. Given vector
$\fb=(f_1,f_2,f_3)\in \br^3$ put
\begin{equation}\label{bij}
\b(\fb)_{ij}=\sum_{k=1}^3b_{ki,j}f_k.
\end{equation}
Define a matrix $\bb(\fb)=(\b(\fb)_{ij})_{ij=1}^3$.

By $\|\bb(\fb)\|$ we denote a norm of the matrix $\bb(\fb)$
associated with Euclidean norm in $\bc^3$. Put
$$S=\{\pb=(p_1,p_2,p_3)\in\br^3: \ p_1^2+p_2^2+p_3^2\leq 1\}
$$
and denote
$$
|\|\bb\||=\sup_{\fb\in S}\|\bb(\fb)\|. $$
\begin{prop}\cite[Proposition 3.3]{MAA}\label{positive}
Let $\Delta$ be a q.q.o. with a Haar state $\t$, then
$|\|\bb\||\leq 1$.
\end{prop}

Let $\D:\bm_2(\bc)\to \bm_2(\bc)\o\bm_2(\bc) $ be a liner operator
with a Haar state $\t$. Then due to Theorem \ref{qc1} $\D$ has a
form \eqref{D3}. Take arbitrary states $\f,\p\in S(\bm_2(\bc))$ and
$\fb,\pb\in S$ be the corresponding vectors (see \eqref{state}).
Then one finds that
\begin{equation}\label{AA}
\Delta^{*}(\varphi\otimes
\psi)(\s_k)=\overset{3}{\underset{i,j=1}{\sum}}b_{ij,k}f_ip_j, \ \ \
k=1,2,3.
\end{equation}

Thanks to Lemma \ref{m2} the functional $\Delta^{*}(\varphi\otimes
\psi)$ is a state if and only if the vector
$$
\textbf{f}_{\Delta^{*}(\varphi,
\psi)}=\bigg(\overset{3}{\underset{i,j=1}{\sum}}b_{ij,1}f_ip_j,
\overset{3}{\underset{i,j=1}{\sum}}b_{ij,2}f_ip_j,\overset{3}{\underset{i,j=1}{\sum}}b_{ij,3}f_ip_j\bigg).
$$
satisfies $\|\fb_{\D^*(\f,\p)}\|\leq 1$.

So, we have the following

\begin{prop}\cite[Proposition 4.1]{MAA}\label{D*} Let $\D:\bm_2(\bc)\to \bm_2(\bc)\o\bm_2(\bc)$ be a liner operator
with a Haar state $\t$. Then $\D^*(\f\o\p)\in S(\bm_2(\bc))$ for
any $\f,\p\in S(\bm_2(\bc))$ if and only if one holds
\begin{equation}\label{D*1}
\overset{3}{\underset{k=1}{\sum}}\bigg|\overset{3}{\underset{i,j=1}{\sum}}b_{ij,k}f_ip_j\bigg|^{2}\leq
1 \ \ \ \textrm{for all} \ \ \fb,\pb\in S.
\end{equation}
\end{prop}

From the proof of Proposition \ref{positive} and the last
proposition we conclude that $|\|\bb\||\leq 1$ holds if and only
if \eqref{D*1} is satisfied.

\begin{rem} Note that characterizations of positive
maps defined on $\bm_2(\bc)$ were considered in \cite{MM2} (see also
\cite{Kos}). Characterization of completely positive mappings from
$\bm_2(\bc)$ into itself with invariant state $\t$ was established
in \cite{RSW} (see also \cite{MM3}).
\end{rem}

Next we would like to recall (see \cite{MAA}) some conditions for
q.q.o. to be the Kadison-Schwarz ones.

Let $\D:\bm_2(\bc)\to \bm_2(\bc)\o\bm_2(\bc) $ be a linear operator
with a Haar state $\t$, then it has a form \eqref{D3}. Now we are
going to find some conditions to the coefficients $\{b_{ml,k}\}$
when $\D$ is a Kadison-Schwarz operator. Given $x=w_0+\wb\s$ and
state $\f\in S(\bm_2(\bc))$ let us denote
\begin{eqnarray}\label{algam}
&& \xb_m=\big(\langle\bx_{m1},\wb\rangle,
\langle\bx_{m2},\wb\rangle,\langle\bx_{m3},\wb\rangle\big),  \ \ f_m=\f(\s_m), \\
\label{algam1}
&&\a_{ml}=\langle\xb_{m},\xb_{l}\rangle-\langle\xb_{l},\xb_{m}\rangle,\
\ \
\g_{ml}=[\xb_{m},\overline{\xb}_{l}]+[\overline{\xb}_{m},\xb_{l}],
\end{eqnarray}
where $m,l=1,2,3$. Here and what follows $[\cdot,\cdot]$ stands for
the usual cross product in $\bc^3$.  Note that here the numbers
$\a_{ml}$ are skew-symmetric, i.e. $\overline{\a_{ml}}=-\a_{ml}$. By
$\pi$ we shall denote mapping $\{1,2,3,4\}$ to $\{1,2,3\}$ defined
by $\pi(1)=2,\pi(2)=3,\pi(3)=1, \pi(4)=\pi(1)$.

Denote
\begin{equation}\label{qu}
\qb(\fb,\wb)=\big(\langle\b(\fb)_1,[\wb,\overline{\wb}]\rangle,\langle\b(\fb)_2,[\wb,\overline{\wb}]\rangle,
\langle\b(\fb)_3,[\wb,\overline{\wb}]\rangle\big),
\end{equation}
where $\b(\fb)_m=\big(\b(\fb)_{m1},\b(\fb)_{m2},\b(\fb)_{m3}\big)$
(see \eqref{bij})

\begin{thm}\cite[Theorem 3.6]{MAA}\label{ks3}
Let $\D:\bm_2(\bc)\to \bm_2(\bc)\o\bm_2(\bc) $ be a Kadison-Schwarz
operator with a Haar state $\t$, then it has the form \eqref{D3} and
the coefficients $\{b_{ml,k}\}$ satisfy the following conditions
\begin{eqnarray}\label{ks11}
\|\wb\|^2\geq i\overset{3}{\underset{m=1}{\sum}} f_m\alpha_{\pi
(m),\pi (m+1)}+\overset{3}{\underset{m=1}{\sum}} \|\xb_{m}\|^2
\end{eqnarray}
\begin{eqnarray}\label{ks2} \bigg\|\qb(\fb,\wb)-i\overset{3}{\underset{m=1}{\sum}}
f_m\g_{\pi(m),\pi(m+1)}-[\xb_{m},\overline{\xb}_{m}]\bigg\|&\leq&
\|\wb\|^2-i\overset{3}{\underset{k=1}{\sum}} f_k\alpha_{\pi (k),\pi
(k+1)}\nonumber\\
&&-\overset{3}{\underset{m=1}{\sum}}\|\xb_{m}\|^2.
\end{eqnarray}
for all $\fb\in S,\wb\in \bc^3$. Here as before
$\xb_{m}=\big(\langle\bx_{m1},\wb\rangle,\langle\bx_{m2},\wb\rangle,\langle\bx_{m3},\wb\rangle\big)$,
$\bx_{ml}=(b_{ml,1},b_{ml,2},b_{ml,3})$ and $\qb(\fb,\wb)$,
$\a_{ml}$ and $\g_{ml}$ are defined in
\eqref{qu},\eqref{algam},\eqref{algam1}, respectively.
\end{thm}

\begin{rem}
The provided characterization with \cite{MM1,RSW} allows us to
construct examples of positive or Kadison-Schwarz operators which
are not completely positive (see next section).
\end{rem}

Now we are going to give a general characterization of KS-operators.
Let us first give some notations. For a given mapping
$\D:\bm_2(\bc)\to \bm_2(\bc)\o\bm_2(\bc) $, by $\D(\s)$ we denote
the vector $(\D(\s_1),\D(\s_2),\D(\s_3))$, and by $\wb\D(\s)$ we
mean the following
\begin{equation} \label{Dww}
\wb\D(\s)=w_1\D(\s_1)+w_2\D(\s_2)+w_3\D(\s_3),
\end{equation}
where $\wb\in\bc^3$. Note that the last equality \eqref{Dww}, due to
the linearity of $\D$, also can be written as $\wb\D(\s)=\D(\wb\s)$.

\begin{thm}\label{ks-G}
Let $\D:\bm_2(\bc)\to \bm_2(\bc)\o\bm_2(\bc) $ be a unital
$*$-preserving linear mapping. Then $\D$ is a KS-operator if and
only if one has
\begin{eqnarray}\label{ksG1}
i[\wb,\overline{\wb}]\D(\s)+(\wb\D(\s))(\overline{\wb}\D(\s))\leq
\id\o\id,
\end{eqnarray}
for all $\wb\in \bc^3$ with $\|\wb\|=1$.
\end{thm}

\begin{proof}
Let $x\in\textit{M}_2(\mathbb{C})$ be an arbitrary element, i.e.
$x=w_0\id+\wb\s.$ Then $x^*=\overline{w_0}\id+\overline{\wb}\s$.
Therefore
\begin{equation*}
x^*x=\big(|w_0|^2+\|\wb\|^2\big)\id+
\big(w_0\overline{\wb}+\overline{w_0}\wb-i\big[\wb,\overline{\wb}\big]\big)\s
\end{equation*}
Consequently, we have
\begin{eqnarray}\label{ks7}
\D(x)=w_0\id\o\id+\wb\D(\s), \ \
\D(x^*)=\overline{w_0}\id\o\id+\overline{\wb}\D(\s) \\
\label{ks8} \D(x^*x)=\big(|w_0|^2+\|\wb\|^2\big)\id\o\id+
\big(w_0\overline{\wb}+\overline{w_0}\wb-i\big[\wb,\overline{\wb}\big]\big)\D(\s)\\
\label{ks9}
\D(x)^*\D(x)=|w_0|^2\id\o\id+\big(w_0\overline{\wb}+\overline{w_0}\wb)\D(\s)+
(\wb\D(\s))(\overline{\wb}\D(\s))
\end{eqnarray} From
\eqref{ks8}-\eqref{ks9} one gets
\begin{eqnarray*}
\D(x^*x)-\D(x)^*\D(x)=\|\wb\|^2\id\o\id-i\big[\wb,\overline{\wb}\big]\D(\s)
-(\wb\D(\s))(\overline{\wb}\D(\s)).
\end{eqnarray*}
So, the positivity of the last equality implies that
$$
\|\wb\|^2\id\o\id-i\big[\wb,\overline{\wb}\big]\D(\s)
-(\wb\D(\s))(\overline{\wb}\D(\s))\geq 0.
$$
Now dividing both sides by $\|\wb\|^2$ we get the required
inequality. Hence, this completes the proof.
\end{proof}

\section{An Example of q.q.o. which is not Kadision-Schwarz one}

In this section we are going to study dynamics of \eqref{V} for a
special class of quadratic operators. Such a class operators
associated with the following matrix $\{b_{ij,k}\}$ given by:
\begin{eqnarray*}
&&b_{11,1}=\varepsilon;\quad     b_{11,2}=0;\quad     b_{11,3}=0;\\
&&b_{12,1}=0;\quad     b_{12,2}=0;\quad     b_{12,3}=\varepsilon;\\
&&b_{13,1}=0;\quad     b_{13,2}=\varepsilon;\quad     b_{13,3}=0;\\
&&b_{22,1}=0;\quad     b_{22,2}=\varepsilon;\quad     b_{22,3}=0;\\
&&b_{23,1}=\varepsilon;\quad     b_{23,2}=0;\quad     b_{23,3}=0;\\
&&b_{33,1}=0;\quad     b_{33,2}=0;\quad     b_{33,3}=\varepsilon;
\end{eqnarray*}
and $b_{ij,k}=b_{ji,k}$.

Via \eqref{D3} we define a liner operator $\Delta_{\ve}$, for which
$\tau$ is a Haar state. In the sequel we would like to find some
conditions to $\ve$ which ensures positivity of $\D_\ve$.

It is easy that for given $\{b_{ijk}\}$ one can find a form of
$\D_\ve$  as follows
\begin{eqnarray}\label{pp1}
\D_{\ve}(x)&=&w_0\id\o\id+\ve\w_1\s_1\o\s_1+\ve\w_3\s_1\o\s_2+\ve\w_2\s_1\o\s_3\nonumber\\
&&\qquad \quad \ \ +\ve\w_3\s_2\o\s_1+\ve\w_2\s_2\o\s_2+\ve\w_1\s_2\o\s_3\nonumber\\
&&\qquad \quad \ \
+\ve\w_2\s_3\o\s_1+\ve\w_1\s_3\o\s_2+\ve\w_3\s_3\o\s_3,
\end{eqnarray}
where as before $x=w_0\id+\wb\s$.

\begin{thm}\label{ex1}
A linear operator $\D_{\ve}$ given by \eqref{pp1} is a q.q.o. if and
only if $|\ve|\leq\frac{1}{3}$.
\end{thm}

\begin{proof} Let $x=w_0\id+\wb\s$ be a positive element from $\bm_2(\mathbb{C})$.
Let us show positivity of the matrix $\D_{\ve}(x)$. To do it, we
rewrite \eqref{pp1} as follows $\D_{\ve}(x)=w_0\id+\ve\mathbf{B}$,
here
\begin{eqnarray*}
\textbf{B}=\left(
        \begin{array}{cccc}
          \w_3 & \w_2-i\w_1 & \w_2-i\w_1 & \w_1-2i\w_3-\w_2 \\
          \w_2+i\w_1 & -\w_3 & \w_1+\w_2 & -\w_2+i\w_1 \\
          \w_2+i\w_1 & \w_1+\w_2 & -\w_3 & -\w_2+i\w_1 \\
          \w_1+2i\w_3-\w_2 & -\w_2-i\w_1 & -\w_2-i\w_1 & \w_3 \\
        \end{array}
        \right),
\end{eqnarray*}
where positivity of $x$ yields that $w_0,\w_1,\w_2,\w_3$ are real
numbers. In what follows, without loss of generality, we may assume
that $w_0=1$, and therefore $\|\wb\|\leq 1$. It is known that
positivity of $\D_\ve(x)$ is equivalent to positivity of the
eigenvalues of $\D_\ve(x)$.

Let us first examine eigenvalues of $\mathbf{B}$. Simple algebra
shows us that all eigenvalues of $\textbf{B}$ can be written as
follows

\begin{eqnarray*}
&&\lambda_1(\wb)=\w_1+\w_2+\w_3+2\sqrt{\w_1^2+\w_2^2+\w_3^2-\w_1\w_2-\w_1\w_3-\w_2\w_3}\\[2mm]
&&\lambda_2(\wb)=\w_1+\w_2+\w_3-2\sqrt{\w_1^2+\w_2^2+\w_3^2-\w_1\w_2-\w_1\w_3-\w_2\w_3}\\[2mm]
&&\lambda_3(\wb)=\lambda_4(\wb)=-\w_1-\w_2-\w_3
\end{eqnarray*}

Now examine maximum and minimum values of the functions $\l_1(\wb),
\l_2(\wb), \l_3(\wb), \l_4(\wb)$ on the ball $\|\wb\|\le 1$.

One can see that
\begin{equation}\label{333}
|\l_3(\wb)|=|\l_4(\wb)|\leq\sum_{k=1}^3|\w_k|\leq\sqrt{3}\sum_{k=1}^3|\w_k|^2\leq
\sqrt{3}
\end{equation}

Note that the functions $\l_3$,$\l_4$ can reach values $\pm\sqrt{3}$
at $\pm(1/\sqrt{3},1/\sqrt{3},1/\sqrt{3})$.

 Now let us rewrite $\l_1(\wb)$ and $\l_2(\wb)$ as
follows
\begin{eqnarray}\label{ex4}
&&\l_1(\wb)=\w_1+\w_2+\w_3+\frac2{\sqrt{2}}\sqrt{3(\w_1^2+\w_2^2+\w_3^2)-(\w_1+\w_2+\w_3)^2}\\
\label{ex5}
&&\l_2(\wb)=\w_1+\w_2+\w_3-\frac2{\sqrt{2}}\sqrt{3(\w_1^2+\w_2^2+\w_3^2)-(\w_1+\w_2+\w_3)^2}
\end{eqnarray}

One can see that
\begin{eqnarray}\label{ex4}
&&\l_k(h\w_1,h\w_2,h\w_3)=h\l_k(\w_1,\w_2,w_3), \ \ \textrm{if} \ \
h\geq 0,\\[2mm]
\label{ex44} &&\l_1(h\w_1,h\w_2,h\w_3)=h\l_2(\w_1,\w_2,w_3), \ \
\textrm{if} \ \ h\leq 0.
\end{eqnarray}
where $k=1,2$. Therefore, the functions $\l_k(\wb)$, $k=1,2$ reach
their maximum and minumum on the sphere $\w_1^2+\w_2^2+\w_3^2=1$
(i.e. $\|\wb\|=1$). Hence, denoting $t=\w_1+\w_2+\w_3$ from
\eqref{ex4}\ and \eqref{ex5} we introduce the following functions
\begin{eqnarray*}
g_1(t)=t+\frac2{\sqrt{2}}\sqrt{3-t^2},\ \ \ \
g_2(t)=t-\frac2{\sqrt{2}}\sqrt{3-t^2}
\end{eqnarray*}
where $|t|\leq\sqrt3$.

One can find that the critical values of $g_1$ are $t=\pm 1$, and
the critical value of $g_2$ is $t=-1$. Consequently, extremal values
of $g_1$ and $g_2$ on $|t|\leq \sqrt{3}$ are the following:

\begin{eqnarray*}
&&\min\limits_{|t|\leq \sqrt{3}}g_1(t)=-\sqrt{3}, \ \
\max\limits_{|t|\leq \sqrt{3}}g_1(t)=3,\\[2mm]
 &&\ \min\limits_{|t|\leq
\sqrt{3}}g_2(t)=-3, \ \ \max\limits_{|t|\leq
\sqrt{3}}g_2(t)=\sqrt{3}.
\end{eqnarray*}
Therefore, from \eqref{ex4},\eqref{ex44} we conclude that
\begin{eqnarray}\label{l-12}
-3\leq\l_k(\wb)\leq 3,\ \ \textrm{for any} \ \ \|\wb\|\leq1, \
k=1,2.
\end{eqnarray}

It is known that for the spectrum of $\id+\ve{\mathbf{B}}$ one has
$$ Sp(\id+\ve\textbf{B})=1+\ve Sp(\textbf{B})$$
Therefore,
$$ Sp(\id+\ve\textbf{B})=\{1+\ve \l_k(\wb): k=\overline{1,4}\}$$
So, if
$$|\ve|\leq\frac{1}{\max\limits_{\|\wb\|\leq1}|\l_k(\wb)|},\quad k=\overline{1,4}$$
then one can see $1+\ve \l_k(\wb)\geq0$ for all $\|\wb\|\leq1$
$k=\overline{1,4}$. This implies that the matrix $\id+\ve\textbf{B}$
is positive for all $\wb$ with $\|\wb\|\leq1$.

Now assume that $\D_\ve$ is positive. Then $\D_\ve(x)$ is positive
whenever $x$ is positive. This means that $1+\ve \l_k(\wb)\geq0$ for
all $\|\wb\|\leq1$ ($k=\overline{1,4}$). From
\eqref{333},\eqref{l-12} we conclude that $|\ve|\leq 1/3$. This
completes the proof.
\end{proof}

\begin{thm}\label{ex2}
Let $\ve=\frac{1}{3}$ then the corresponding q.q.o. $\D_{\ve}$ is
not KS-operator.
\end{thm}

\begin{proof}
It is enough to show the dissatisfaction of \eqref{ks2} at some
values of $\wb$ ($\|\wb\| \leq1$) and $\fb=(f_1,f_1,f_2)$.

Assume that $\fb=(1,0,0)$, then a little algebra shows that
\eqref{ks2} reduces to the following one
\begin{eqnarray}\label{ABCD}
\sqrt{A+B+C}\leq D
\end{eqnarray}
where
\begin{eqnarray*}
A&=&|\ve(\overline{\w}_2\w_3-\overline{\w}_3\w_2)-i\ve^2(2\overline{\w}_2\w_3-2|\w_1|^2-\overline{\w}_2\w_1+\overline{\w}_1\w_2-\overline{\w}_1\w_3+\overline{\w}_3\w_1)|^2\\[2mm]
B&=&|\ve(\overline{\w}_1\w_2-\overline{\w}_2\w_1)-i\ve^2(2\overline{\w}_1\w_2-2|\w_3|^2-\overline{\w}_1\w_3+\overline{\w}_3\w_1-\overline{\w}_3\w_2+\overline{\w}_2\w_3)|^2\\[2mm]
C&=&|\ve(\overline{\w}_3\w_1-\overline{\w}_1\w_3)-i\ve^2(2\overline{\w}_3\w_1-2|\w_2|^2-\overline{\w}_3\w_2+\overline{\w}_2\w_3-\overline{\w}_2\w_1+\overline{\w}_1\w_2)|^2\\[2mm]
D&=&(1-3|\ve|^2)(|\w_1|^2+|\w_2|^2+|\w_3|^2)\\
&&-i\ve^2(\overline{\w}_3w_2-\overline{\w}_2\w_3+\overline{\w}_2\w_1-\overline{\w}_1\w_2+
\overline{\w}_1\w_3-\overline{\w}_3\w_1)
\end{eqnarray*}

Now choose $\wb$ as follows:
$$\w_1=-\frac{1}{9};\quad\quad \w_2=\frac{5}{36};\quad\quad \w_3=\frac{5i}{27}$$
Then calculations show that
\begin{eqnarray*}
A&=&\frac{9594}{19131876};\quad\quad
B=\frac{19625}{86093442};\\
C&=&\frac{1625}{3779136};\quad\quad D=\frac{589}{17496}.
\end{eqnarray*}
Hence, we find
\begin{eqnarray*}
\sqrt{\frac{9594}{19131876}+
\frac{19625}{86093442}+\frac{1625}{3779136}}>\frac{589}{17496}\\
\end{eqnarray*}
which means that \eqref{ABCD} is not satisfied. Hence, $\D_\ve$ is
not a KS-operator at $\ve=1/3$.
\end{proof}

Recall that a linear operator $T: \bm_k(\bc)\to \bm_m(\bc)$ is
\textit{completely positive} if for any positive matrix
$\big(a_{ij})\big)_{i,j=1}^n\in \bm_k(\bm_n(\bc))$ the matrix
$\big(T(a_{ij})\big)_{i,j=1}^n$ is positive for all $n\in\bn$. Now
we are interested when the operator $\Delta_\ve$ is completely
positive. It is known \cite{Choi} that the complete positivity of
$\Delta_\ve$ is equivalent to the positivity of the following matrix
\begin{eqnarray*}
\hat\Delta_\ve=\left(
  \begin{array}{cc}
    \D_\ve(e_{11}) &  \D_\ve(e_{12}) \\
     \D_\ve(e_{21}) &  \D_\ve(e_{22}) \\
  \end{array}
\right)
\end{eqnarray*}
here $e_{ij}$ ($i,j=1,2$) are the standard matrix units in
$\bm_2(\bc)$.

From \eqref{pp1} one can calculate that

\begin{eqnarray*}
&& \D_\ve(e_{11})=\frac{1}{2}\id\o\id+\ve B_{11}, \ \
\D_\ve(e_{22})=\frac{1}{2}\id\o\id-\ve B_{11}\\[2mm]
&& \D_\ve(e_{12})=\ve B_{12}, \ \ \D_\ve(e_{21})=\ve B_{12}^*
\end{eqnarray*}
where
\begin{eqnarray*}
B_{11}=\left(
          \begin{array}{cccc}
            \frac{1}{2} & 0 & 0 & -i \\
            0 & -\frac{1}{2} & 0 & 0 \\
            0 & 0 & -\frac{1}{2} & 0 \\
            i & 0 & 0 & \frac{1}{2} \\
          \end{array}
        \right), \ \
B_{12}=\left(
         \begin{array}{cccc}
           0 & 0 & 0 & \frac{1-i}{2} \\
           i & 0 & \frac{1+i}{2} & 0 \\
           i & \frac{1+i}{2} & 0 & 0 \\
           \frac{1-i}{2} & -i & -i & 0 \\
         \end{array}
       \right)
\end{eqnarray*}
Hence, we find
\begin{eqnarray*}
2\hat\Delta_\ve=\id_8+\ve\mathbb{B}
\end{eqnarray*}
where $\id_8$ is the unit matrix in $\bm_8(\bc)$ and
\begin{eqnarray*}
\mathbb{B}=\left(
             \begin{array}{cccccccc}
               1 & 0 & 0 & -2i & 0 & 0 & 0 & 1-i \\
               0 & -1 & 0 & 0 & 2i & 0 & 1+i & 0 \\
               0 & 0 & -1 & 0 & 2i & 1+i & 0 & 0 \\
               2i & 0 & 0 & 1 & 1-i & -2i & -2i & 0 \\
               0 & -2i & -2i & 1+i & -1 & 0 & 0 & 2i \\
               0 & 0 & 1-i & 2i & 0 & 1 & 0 & 0 \\
               0 & 1-i & 0 & 2i & 0 & 0 & 1 & 0 \\
               1+i & 0 & 0 & 0 & -2i & 0 & 0 & -1 \\
             \end{array}
           \right)
\end{eqnarray*}
So, the matrix $\hat\Delta_\ve$ is positive if and only if
\begin{eqnarray*}
|\ve|\leq \frac{1}{\l_{\max}(\mathbb{B})},
\end{eqnarray*}
where $\l_{\max}(\mathbb{B})=\max\limits_{\l\in
Sp(\mathbb{B})}|\l|$.

One can easily calculate that $\l_{\max}(\mathbb{B})=3\sqrt{3}$.
Therefore, we have the following

\begin{thm} Let $\D_{\ve}: \bm_2(\bc)\to \bm_2(\bc)\o\bm_2(\bc)$ be given by
\eqref{pp1}. Then $\D_\ve$ is completely positive if and only if
$|\ve|\leq\frac{1}{3\sqrt{3}}$.
\end{thm}

\section{Dynamics of $\Delta_\ve$}

Let $\D$ be a q.q.o. on $\bm_2(\bc)$. Let us consider the
corresponding quadratic operator defined by
$V_\D(\f)=\D^*(\f\o\f)$, $\f\in S(\bm_{2}(\bc))$. From Theorem
\ref{qc1} one can see that the defined operator $V_\D$ maps
$S(\bm_{2}(\bc))$ into itself if and only if $\||\bb\||\leq 1$ or
equivalently \eqref{D*1} holds. From \eqref{AA} we find that
\begin{eqnarray}\label{VD}
V_\Delta(\varphi)(\s_k)=\overset{3}{\underset{i,j=1}{\sum}}
b_{ij,k}f_if_j, \ \ \fb\in S.
\end{eqnarray}
Here as before $S=\{\fb=(f_1,f_2fp_3)\in\br^3: \
f_1^2+f_2^2+f_3^2\leq 1\}.$

So, \eqref{VD} suggests us to consider of the following nonlinear
operator $V:S\to S$ defined by
\begin{equation}\label{V}
V(\textbf{f})_{k}=\overset{3}{\underset{i,j=1}{\sum}}b_{ij,k}f_{i}f_{j},
  \ \ \  k=1,2,3.
\end{equation}
where $\fb=(f_1,f_2,f_3)\in S$.

It is worth to mention that uniqueness of the fixed point (i.e.
$(0,0,0)$) of the operator given by \eqref{V} was investigated in
\cite[Theorem 4.4]{MAA}.

In this section, we are going to study dynamics of the quadratic
operator $V_\ve$ corresponding to $\D_\ve$ (see \eqref{pp1}), which
has the following form
\begin{eqnarray}\label{Ve}
\left\{\begin{array}{c}
         V_\ve(f)_1=\ve(f_1^2+2f_2f_3) \\
         V_\ve(f)_2=\ve(f_2^2+2f_1f_3) \\
         V_\ve(f)_3=\ve(f_3^2+2f_1f_2)

       \end{array}
\right.
\end{eqnarray}

Let us first find some condition on $\ve$ which ensures \eqref{D*1}.

\begin{lem}\label{D22} Let $V_\ve$ be given by \eqref{Ve}. Then
$V_\ve$ maps $S$ into itself if and only if $|\ve|\leq\frac1
{\sqrt3}$ is satisfied.
\end{lem}

\begin{proof} "If" part. Assume that $V_\ve$ maps $S$ into itself.
Then \eqref{D*1} is satisfied. Take
$\fb=(1/\sqrt{3},1/\sqrt{3},1/\sqrt{3})$, $\pb=\fb$. Then from
\eqref{D*1} one finds
\begin{eqnarray*}
\sum_{k=1}^3 \bigg|\sum_{i,j=1}^3 b_{ij,k}f_i p_j\bigg|^2 =3\ve^2
\leq 1
\end{eqnarray*}
which yields $|\ve|\leq 1/\sqrt{3}$.

"only if" part. Assume that $|\ve|\leq 1/\sqrt{3}$. Take any
$\fb=(f_1,f_2,f_3),\pb=(p_1,p_2,p_3)\in S$. Then one finds
\begin{eqnarray*}
\sum_{k=1}^3 \bigg|\sum_{i,j=1}^3 b_{ij,k}f_i p_j\bigg|^2 &=&\ve^2
(|f_1p_1+f_3p_2+f_2p_3|^2+|f_3p_1+f_2p_2+f_1p_3|^2\\
&&+|f_2p_1+f_1p_2+f_3p_3|^2)\\[2mm]
&\leq&\ve^2((f_1^2+f_2^2+f_3^2)(p_1^2+p_2^2+p_3^2)+(f_3^2+f_2^2+f_1^2)(p_1^2+p_2^2+p_3^2)\\
&&+(p_1^2+p_2^2+p_3^2)(f_2^2+f_1^2+f_3^2))\\[2mm]
&\leq&\ve^2(1+1+1)=3\ve^2\leq 1.
\end{eqnarray*}
This completes the proof.
\end{proof}

\begin{rem} We stress that condition \eqref{D*1}
is necessary for $\D$ to be a positive operator. Namely, from
Theorem \ref{ex1} and Lemma \ref{D22} we conclude that if
$\ve\in(\frac13,\frac{1}{\sqrt{3}}]$ then the operator $\D_\ve$ is
not positive, while \eqref{D*1} is satisfied.\\
\end{rem}

In what follows, to study dynamics of $V_\ve$ we assume
$|\ve|\leq\frac{1}{\sqrt{3}}$.
 Recall that a vector $\fb\in S$ is a fixed point
of $V_\ve$ if $V_\ve(\fb)=\fb$. Clearly $(0,0,0)$ is a fixed point
of $V_\ve$. Let us find others. To do it, we need to solve the
following equation
\begin{eqnarray}\label{vfp}
\left\{\begin{array}{c}
         \ve(f_1^2+2f_2f_3)=f_1 \\
         \ve(f_2^2+2f_1f_3)=f_2 \\
         \ve(f_3^2+2f_1f_2)=f_3
       \end{array}
\right.
\end{eqnarray}

 We have the following
\begin{prop}\label{fix}
If $|\ve|<\frac1{\sqrt3}$ then $V_\ve$ has a unique fixed point
$(0,0,0)$ in $S$.  If $|\ve|=\frac1{\sqrt3}$ then $V_\ve$ has the
following fixed points $(0,0,0)$ and
$(\pm\frac1{\sqrt3},\pm\frac1{\sqrt3},\pm\frac1{\sqrt3})$ in $S$.
\end{prop}

\begin{proof} It is clear that $(0,0,0)$ is a fixed point of
$V_\ve$. If $f_k=0$, for some $k\in\{1,2,3\}$ then due to $|\ve|\leq
\frac1{\sqrt3}$, one can see that only solution of \eqref{vfp}
belonging to $S$ is $f_1=f_2=f_3=0$. Therefore, we assume that
$f_k\neq 0$ ($k=1,2,3$). So, from \eqref{vfp} one finds
  \begin{eqnarray}\label{pr1}
\left\{
    \begin{array}{c}
             \frac{{f_1}^2+2f_2f_3}{{f_2}^2+2f_1f_3}=\frac{f_1}{f_2} \\
             \frac{{f_1}^2+2f_2f_3}{{f_3}^2+2f_1f_2}=\frac{f_1}{f_3} \\
             \frac{{f_2}^2+2f_1f_3}{{f_3}^2+2f_1f_2}=\frac{f_2}{f_3}\\
           \end{array}
\right.
  \end{eqnarray}
  Denoting
  \begin{equation}\label{pr1a}
  x=\frac{f_1}{f_2}, \ y=\frac{f_1}{f_3}, \ z=\frac{f_2}{f_3}
  \end{equation}
  From \eqref{pr1} it follows that
  \begin{eqnarray}\label{pr2}
\left\{
  \begin{array}{ccc}
           x\bigg(\frac{x\big(1+\frac{2}{xy}\big)}{1+\frac{2x}{z}}-1\bigg)=0  \\
           y\bigg(\frac{y\big(1+\frac{2}{xy}\big)}{1+2yz}-1\bigg)=0 \\
           z\bigg(\frac{z\big(1+\frac{2x}{z}\big)}{1+2yz}-1\bigg)=0\\
         \end{array}
\right.
  \end{eqnarray}

According to our assumption $x,y,z$ are nonzero, so from \eqref{pr2}
one gets
 \begin{eqnarray}\label{pr3}
  \left\{
  \begin{array}{ccc}
  \frac{x\big(1+\frac{2}{xy}\big)}{1+\frac{2x}{z}}=1 \\
  \frac{y\big(1+\frac{2}{xy}\big)}{1+2yz}=1  \\
  \frac{z\big(1+\frac{2x}{z}\big)}{1+2yz}=1
  \end{array}
  \right.
  \end{eqnarray}
where $2x\neq-z$ and $2yz\neq-1$.

Dividing the second equality of \eqref{pr3} to the first one of
\eqref{pr3} we find
\begin{eqnarray*}
  \frac{y\big(1+\frac{2x}{z}\big)}{x(1+2yz)}=1
\end{eqnarray*}
which with $xz=y$ yields
\begin{eqnarray*}
  y+2x^2=x+2y^2.
\end{eqnarray*}
Simplifying the last equality one gets
\begin{eqnarray*}
  (y-x)(1-2(y+x))=0.
\end{eqnarray*}
This means that either $y=x$ or $x+y=\frac{1}{2}$.

Assume that $x=y$. Then from $xz=y$, one finds $z=1$. Moreover, from
the second equality of \eqref{pr3} we have $y+\frac{2}{y}=1+2y$. So,
$y^2+y-2=0$ therefore, the solutions of the last one are $y_1=1,
y_2=-2$. Hence, $x_1=1, x_2=-2$.

Now suppose that $x+y=\frac{1}{2}$, then $x=\frac{1}{2}-y$. We note
that $y\neq 1/2$, since $x\neq 0$. So, from the second equality of
\eqref{pr3} we find
$$y+\frac{4}{1-2y}=1+\frac{4y^2}{1-2y}.
$$
So, $2y^2-y-1=0$ which yields the solutions $y_3=-\frac{1}{2},
y_4=1$. Therefore,  we obtain $x_3=1$, $z_3=-\frac{1}{2}$ and
$x_4=-\frac{1}{2}$, $z_4=-2$.

Consequently, solutions of \eqref{pr3} are the following ones
$$(1,1,1), \
\big(1,-\frac{1}{2},-\frac{1}{2}\big),\
\big(-\frac{1}{2},1,-2\big),\  (-2,-2,1).$$

Now owing to \eqref{pr1a} and we need to solve the following
equations
\begin{eqnarray}\label{ffk}
  \left\{
  \begin{array}{l}
  \frac{f_1}{f_2}=x_k, \\
  \frac{f_2}{f_3}=z_k ,
  \end{array}
  \right. \ \ \ k=\overline{1,4}.
\end{eqnarray}
According to our assumption $f_k\neq 0$, therefore we consider cases
when $x_kz_k\neq 0$.

Now let us start to consider several cases:

{\sc case 1.} Let $x_2=1$, $z_2=1$. Then from \eqref{ffk} one gets
$f_1=f_2=f_3$. So, from \eqref{vfp} we find $3\varepsilon
{f_1}^2=f_1$, i.e. $f_1=\frac{1}{3\varepsilon}$. Now taking into
account ${f_1}^2+{f_2}^2+{f_3}^2\leq1$ one gets
$\frac{1}{3\varepsilon^2}\leq1$. From the last inequality we have
$|\varepsilon|\geq\frac{1}{\sqrt{3}}$. Due to Lemma \ref{D22} the
operator $V_\ve$ is well defined iff $|\ve|\leq\frac{1}{\sqrt{3}}$,
therefore, one gets $|\varepsilon|=\frac{1}{\sqrt{3}}$. Hence, in
this case a solution is
$\big(\pm\frac{1}{\sqrt{3}};\pm\frac{1}{\sqrt{3}};\pm\frac{1}{\sqrt{3}}\big)$.

{\sc case 2.} Let $x_2=1$, $z_2=-1/2$. Then from \eqref{ffk} one
finds $f_1=f_2, 2f_2=-f_3$. Substituting the last ones to
\eqref{vfp} we get $f_1+3{f_1}^2\varepsilon=0$. Then, we have
$f_1=-\frac{1}{3\varepsilon}, f_2=-\frac{1}{3\varepsilon},
f_3=\frac{2}{3\varepsilon}$. Taking into account
${f_1}^2+{f_2}^2+{f_3}^2\leq1$ we find
$\frac{1}{9\varepsilon^2}+\frac{4}{9\varepsilon^2}+\frac{1}{9\varepsilon^2}\leq1$.
This means $|\varepsilon|\geq\sqrt{\frac{2}{3}}$, due to Lemma
\ref{D22} in this case the operator $V_\ve$ is not well defined,
therefore, we conclude that there is not a fixed point of $V_\ve$
belonging to $S$.

Using the same argument for the rest cases we conclude the absence
of solutions. This shows that if $|\ve|<1/\sqrt{3}$ the operator
$V_\ve$ has unique fixed point in $S$. If $|\ve|=1/\sqrt{3}$, then
$V_\ve$ has three fixed points belonging to $S$. This completes the
proof.
\end{proof}

Now we are going to study dynamics of operator $V_\ve$.

\begin{thm} Let $V_\ve$ be given by \eqref{Ve}. Then the following assertions hold true:
\begin{enumerate}
\item[(i)] if $|\ve|<1/\sqrt{3}$, then for any $\fb\in S$ one has $V_{\ve}^n(\fb)\to (0,0,0)$ as
$n\to\infty$.
\item[(ii)] if $|\ve|=1/\sqrt{3}$, then for any $\fb\in S$ with
$\fb\notin\big\{(\pm\frac1{\sqrt3},\pm\frac1{\sqrt3},\pm\frac1{\sqrt3})\big\}$
one has $V_{\ve}^n(\fb)\to (0,0,0)$ as $n\to\infty$.
\end{enumerate}
\end{thm}

\begin{proof}
Let us consider the following function
$\rho(\fb)=f_1^2+f_2^2+f_3^2$. Then we have
\begin{eqnarray*}
\rho(V_{\ve}(\fb))&=&\ve^2\big((f_1^2+2f_2f_3)^2+(f_2^2+2f_1f_3)^2+(f_3^2+2f_1f_2)^2\big)\\
&\leq&\ve^2\big(f_1^2+2|f_2||f_3|+f_2^2+2|f_1||f_3|+f_3^2+2|f_1||f_2|\big)\\
&\leq&\ve^2\big(f_1^2+f_2^2+f_3^2+f_2^2+f_1^2+f_3^2+f_3^2+f_1^2+f_2^2)\\
&=&3\ve^2(f_1^2+f_2^2+f_3^2)= 3\ve^2\rho(\fb)
\end{eqnarray*}

 This means
\begin{eqnarray}\label{ex7}
\rho(V_{\ve}(\fb))\leq3\ve^2\rho(\fb).
\end{eqnarray}

Due to $\ve^2\leq\frac13$ from \eqref{ex7} one finds
\begin{eqnarray*}
\rho(V_{\ve}^{n+1}(\fb))\leq\rho(V_{\ve}^n(\fb)),
\end{eqnarray*}
which yields that the sequence $\{\rho(V_{\ve}^n(\fb))\}$ is
convergent. Next we would like to find the limit of
$\{\rho(V_{\ve}^n(\fb))\}$.

(i). First we assume that $|\ve|<\frac{1}{\sqrt{3}}$, then from
\eqref{ex7} we obtain
\begin{eqnarray*}
\rho(V_{\ve}^n(\fb))\leq3\ve^2\rho(V_{\ve}^{n-1}(\fb))\leq\cdots\leq(3\ve^2)^n\rho(\fb).
\end{eqnarray*}

This yields that $\rho(V_{\ve}^n(\fb))\rightarrow0$ as
$n\rightarrow\infty$, for all $\fb\in S$.

(ii). Now let $|\ve|=\frac{1}{\sqrt{3}}$. Then consider two distinct
subcases.

{\sc Case (a)}. Let $f_1^2+f_2^2+f_3^2<1$ and denote
$d=f_1^2+f_2^2+f_3^2$. Then one gets
\begin{eqnarray*}
\rho(V_{\ve}(\fb))&\leq& \ve^2\big((f_1^2+2|f_2||f_3|)^2+(f_2^2+2|f_1||f_3|)^2+(f_3^2+2|f_1||f_2|)^2\big)\\
&\leq&\ve^2\big((f_1^2+f_2^2+f_3^2)^2+(f_2^2+f_1^2+f_3^2)^2+(f_3^2+f_1^2+f_2^2)^2\big)\\
&=&3\ve^2d^2=dd=d\rho(\fb).
\end{eqnarray*}
Hance,  we have $\rho(V_{\ve}(\fb))\leq d\rho(\fb)$. This means
$\rho(V_{\ve}^n(\fb))\leq d^n\rho(\fb)\rightarrow0 $. Hence,
$V_\ve^n(\fb)\to 0$ as $n\rightarrow\infty$.

{\sc Case (b).} Now take $f_1^2+f_2^2+f_3^2=1$ and assume that $\fb$
is not a fixed point. Therefore, we may assume that $f_i\neq f_j$
for some $i\neq j$, otherwise from Lemma \ref{fix} one concludes
that $\fb$ is a fixed point. Hence, from \eqref{Ve} one finds
\begin{eqnarray*}
V_\ve(\fb)_1=\ve(f_1^2+2f_2f_3)=\ve(1-f_2^2-f_3^2+2f_2f_3)=\ve(1-(f_2-f_3)^2).
\end{eqnarray*}
Similarly, one gets
\begin{eqnarray*}
&&V_\ve(\fb)_2=\ve(1-(f_1-f_3)^2),\\
&&V_\ve(\fb)_3=\ve(1-(f_1-f_2)^2).
\end{eqnarray*}
It is clear that $|V_\ve(\fb)_k|\leq |\ve|$ ($k=1,2,3$). According
to our assumption $f_i\neq f_j$ ($i\neq j$) we conclude that one
of $|V_\ve(\fb)_k|$ is strictly less than $\frac1{\sqrt3}$, this
means $V_\ve(\fb)_1^2+V_\ve(\fb)_2^2+V_\ve(\fb)_3^2<1$. Therefore,
from the case (a), one gets that $V_\ve^n(\fb)\to 0$ as
$n\rightarrow\infty$.

\end{proof}

\section*{Acknowledgement} The authors acknowledges the MOHE Grant FRGS11-022-0170. The first
named author thanks the Junior Associate scheme of the Abdus Salam
International Centre for Theoretical Physics, Trieste, Italy. The
authors would like to thank to an anonymous referee whose useful
suggestions and comments improve the content of the paper.


\begin{thebibliography}{92}

\bibitem{Bh} S.J. Bhatt, Stinespring representability and Kadison's Schwarz inequality in
non-unital Banach star algebras and applications, \textit{Proc.
Indian Acad. Sci. (Math. Sci.)}, {\bf 108} (1998), 283--303.

\bibitem{BD} R. Bhatia, C. Davies, More Operator Versions of the Schwarz Inequality,
\textit{Commun. Math. Phys.} {\bf 215} (2000), 239--244.

\bibitem{BS} R. Bhatia, R. Sharma, Some inequalities for positive linear
maps, \textit{Linear Alg. Appl.} {\bf 436} (2012), 1562--1571.

\bibitem{BR} O. Bratteli and D. W. Robertson, {\it Operator algebras and quantum
statistical mechanics.} I, Springer, New York-–Heidelberg-–Berlin
1979.

\bibitem{Choi} M.-D. Choi, Completely Positive Maps on Complex Matrices,
\textit{Linear Algebra Appl.} {\bf 10} (1975), 285-–290.

\bibitem{GM1} N. N. Ganikhodzhaev, F. M. Mukhamedov, On quantum quadratic
stochastic processes, and some ergodic theorems for such processes,
{\it Uzb. Matem. Zh.} 1997, no. 3, 8–20. (Russian)

\bibitem{GM2} N.N. Ganikhodzhaev, F. M. Mukhamedov, Ergodic properties
of quantum quadratic stochastic processes, \textit{Izv. Math.} {\bf
65} (2000), 873–890.

\bibitem{GMR} R. Ganikhodzhaev, F. Mukhamedov, U. Rozikov, Quadratic
stochastic operators and processes: results and open problems,
\textit{Inf. Dim. Anal. Quantum Probab. and Related Topics} {\bf
14}(2011), 279–-335.

\bibitem{Gr1} U. Groh, The peripheral point spectrum of Schwarz operator on
$C^*$-algebras, \textit{Math. Z.} {\bf 176} (1981), 311--318.

\bibitem{Gr2} U. Groh, Uniform ergodic theorems for
identity preserving Schwarz maps on $W^*$-algebras, \textit{J.
Operator Theory} {\bf 11} (1984), 395--404.

\bibitem{FS} U. Franz, A. Skalski, On ergodic properties of convolution
operators associated with compact quantum groups,
{\it Colloq. Math.} {\bf 113} (2008), 13--23.

\bibitem{Kos} A. Kossakowski, A Class of Linear Positive Maps in Matrix
Algebras, {\it Open Sys. \& Information Dyn.} {\bf 10}(2003)
213--220.

\bibitem{Ma} W.A. Majewski, On non-completely
positive quantum dynamical maps on spin chains, {\it J. Phys. A: Math. Gen.} {\bf 40} (2007) 11539-–11545.

\bibitem{MM1} W.A. Majewski, M. Marciniak, On a characterization of positive
maps, {\it J. Phys. A: Math. Gen.} {\bf 34} (2001) 5863-–5874.

\bibitem{MM2} W.A. Majewski, M. Marciniak, On the structure of positive maps
between matrix algebras, In book:
 \textit{Noncommutative harmonic analysis with applications to probability},
 249--263, Banach Center Publ., 78, Polish Acad. Sci. Inst. Math., Warsaw, 2007.


\bibitem{MM3} W.A. Majewski, M. Marciniak, Decomposability of extremal
positive unital maps on $M_2({\Bbb C})$. In book: M. Bozejko, (ed.)
et al., Quantum probability, Banach Center Publications {\bf
73}(2006), 347--356.


\bibitem{M2} F.M. Mukhamedov, On ergodic properties of discrete quadratic
dynamical system on $C^*$-algebras.  {\it Method of Funct. Anal. and
Topology}, {\bf 7}(2001), No.1, 63--75.

\bibitem{M3} F.M. Mukhamedov, On decomposition of quantum quadratic
stochastic processes into layer-Markov processes defýned on von
Neumann algebras, {\it Izvestiya Math.} {\bf 68}(2004), No.5.
1009--1024.

\bibitem{MA} F. Mukhamedov, A. Abduganiev, On description of bistochastic Kadison-Schwarz
operators on $M_2(\mathbb{C})$, \textit{Open Sys. \& Inform. Dynam.}
{\bf 17}(2010), 245-253.


\bibitem{MAA} F. Mukhamedov, H. Akin, S. Temir, A. Abduganiev, On quantum
quadratic operators on $M_2(\mathbb{C})$ and their dynamics,
\textit{Jour. Math. Anal. Appl.} {\bf 376}(2011), 641--655.


\bibitem{N} M.A. Nielsen, I.L. Chuang, {\it Quantum Computation and Quantum Information},
Cambridge University
Press, Cambridge, 2000.

\bibitem{Rob1} A. G. Robertson, A Korovkin theorem for Schwarz maps on $C^*$-algebras, {\it Math. Z.} {\bf 156}(1977),
205--206.

\bibitem{Rob} A. G. Robertson, Schwarz inequalities and the decomposition of positive maps on $C^*$-algebras {\it Math. Proc. Camb. Philos. Soc.} {\bf 94}(1983),
291--296.

\bibitem{RSW} M.B. Ruskai, S. Szarek, E. Werner, An analysis of completely positive
trace-preserving maps on $M_2$, {\it Lin. Alg. Appl.} {\bf 347}
(2002) 159–-187.

\bibitem{St} E. Stormer, Positive linear maps of operator algebras, {\it Acta Math.} {\bf 110}(1963), 233--278.

\bibitem{W} S.L. Woronowicz, Compact matrix pseudogroups. {\it Comm. Math. Phys.}
{\bf 111} (1987), 613-–665.





\end{thebibliography}
\end{document}